\documentclass{article}
\usepackage{amsthm}
\usepackage{amssymb}
\usepackage{graphicx}
\usepackage{epsfig}
\usepackage{varwidth}
\usepackage{algpseudocode}
\usepackage{color}
\usepackage{tikz}
\usepackage{verbatim}
\usepackage{algpseudocode}

\usetikzlibrary{arrows}

\newtheorem{thm}{Theorem}[section]
\newtheorem{cor}[thm]{Corollary}
\newtheorem{lem}[thm]{Lemma}

\theoremstyle{definition}

\theoremstyle{remark}

\begin{document}

\title{Gray Codes and Overlap Cycles for Restricted Weight Words}%
\author{Victoria Horan\thanks{\texttt{victoria.horan.1@us.af.mil}} \\  Air Force Research Laboratory \\ Information Directorate \\ \\ Glenn Hurlbert\thanks{\texttt{hurlbert@asu.edu}} \\ School of Mathematical and Statistical Sciences \\ Arizona State University}%


\maketitle

\begin{abstract}
    A Gray code is a listing structure for a set of combinatorial objects such that some consistent (usually minimal) change property is maintained throughout adjacent elements in the list.  While Gray codes for $m$-ary strings have been considered in the past, we provide a new, simple Gray code for fixed-weight $m$-ary strings.  In addition, we consider a relatively new type of Gray code known as overlap cycles and prove basic existence results concerning overlap cycles for fixed-weight and weight-range $m$-ary words.
\end{abstract}


\let\thefootnote\relax\footnote{Approved for public release; distribution unlimited:  88ABW-2014-0956, 07 Mar 2014}

\section{Introduction}

Gray codes were originally developed by Frank Gray \cite{Gray} as a method of listing binary $n$-tuples so that successive words differ in only one position. The term Gray code has now come to mean a listing of a set $\mathcal{C}$ of combinatorial objects in which successive words differ in some predefined manner, usually a consistent minimal change.  Since Gray's original code, these listings have been studied extensively and have seen use in many different applications such as rotary encoders \cite{rotary} and error detection and correction \cite{errors}.

Define the set of all strings of length $n$ over the alphabet $\{0,1, \ldots , m-1\}$ to be $\mathcal{B}(m,n)$, the set of all \textbf{$m$-ary strings of length $n$}.  Given a string $x=x_1x_2 \ldots x_n$, the \textbf{weight} of $x$ is defined to be $wt(x)=\sum_{i=1}^n x_i$.  Then the set of all weight $k$ $m$-ary strings of length $n$ is denoted $\mathcal{B}_k(m,n)$.  We are interested in two kinds of Gray codes.  The first kind requires the set of fixed-weight strings to be ordered such that they change in exactly two positions (Section \ref{GC}), and the second kind are $s$-overlap cycles, which require that adjacent strings overlap in a specific way (Section \ref{OCycles}).

\section{Gray Codes}\label{GC}

Gray codes for $m$-ary words have been studied extensively in the past.  The first main theorem considers all $m$-ary words of length $n$.

\begin{thm}
    \emph{\cite{Er}}  For every $n,m \in \mathbb{Z}^+$, there exists a Gray code listing for $\mathcal{B}(m,n)$ so that each word differs from its successor in exactly one position.
\end{thm}

For fixed-weight $m$-ary words of length $n$, the following theorem is known and published in \cite{Walsh}, however we provide a new and simpler algorithm.

\begin{thm}
    For every $n,m,k \in \mathbb{Z}^+$, there exists a Gray code listing for $\mathcal{B}_k(m,n)$ in which successive words differ in at most two positions.
\end{thm}

Note that a change in two positions is best possible, as a change in just one would alter the weight of the word.  In this section we will first present our algorithm, then provide an example, and finally will prove that our algorithm is correct.

\subsection{The New Algorithm}

We would like to produce a Gray code for the set of fixed-weight $m$-ary strings in which successive elements differ in at most two positions.  First, we define a few simple functions that will be used in the algorithm.  Given a list $L$, \textsf{Rev}$(L)$ produces the list in reversed order.  The second function needed is the exponent function, defined as follows. $$\hbox{\textsf{Expo}}(L,e) = \left\{
                                                                                     \begin{array}{ll}
                                                                                       L, & \hbox{if $e\equiv 0 \bmod 2$;} \\
                                                                                       \hbox{\textsf{Rev}}(L), & \hbox{otherwise.}
                                                                                     \end{array}
                                                                                   \right.$$
Finally, the function \textsf{Pref}$(a,L)$ adds the prefix $a$ to every string in list $L$.  This operation will also be denoted $a \oplus L$.  Now we can provide the following algorithm, borrowing the reflection idea from Gray's original algorithm for binary words \cite{Gray}.  Our algorithm is presented in Figure \ref{euclid}.

\begin{figure}
\begin{algorithmic}[1]
\Procedure{FWM}{$m, n, k$}\Comment{$m$-ary words of length $n$ and weight $k$}
\State $L\gets [ \hbox{ }]$
\If{$(m-1)n \geq k$ \textbf{and} $k \geq 0$}
\If{$(m-1)n=0$}
\State $L \gets [[\hbox{ }]]$
\EndIf\label{euclidendwhile}
\For{$i=0:\min(m-1,k)$}
\State $M \gets$ \textsf{Pref}($i, \hbox{\textsf{Expo}}(FWM(m,n-1,k-i),i))$
\State $L \gets L, M$ \Comment{Append $M$ to $L$}
\EndFor\label{euclidendwhile}
\EndIf\label{euclidendwhile}
\State \textbf{return} $L$
\EndProcedure
\end{algorithmic}
\caption{Fixed-Weight $m$-ary Gray Code Algorithm}\label{euclid}
\end{figure}

\subsection{Example}

We will work through an example of the algorithm when we run FWM$(3,4,5)$.

\begin{itemize}

\item We begin with $L$, an empty list.

\item Starting with $i=0$, we need to determine $$M = 0 \oplus \hbox{FWM}(3,3,5) = \hbox{\textsf{Pref}}(0,\hbox{\textsf{Expo}}(\hbox{FWM}(3,3,5),0)).$$

\begin{itemize}

\item  We begin by determining the sublist, FWM$(3,3,5)$. $$\hbox{FWM}(3,3,5) = [0 \oplus \hbox{FWM}(3,2,5),1 \oplus \hbox{\textsf{Rev}}(\hbox{FWM}(3,2,4)),2 \oplus \hbox{FWM}(3,2,3)].$$

\item  For these sublists, we have the following. $$\begin{array}{c|cc} \hline 1 & 2 & 2 \\ \hline 2 & 1 & 2 \\ 2 & 2 & 1 \end{array}$$

\end{itemize}

\item  Now we must precede this sublist by $0$ to determine $M$. $$\begin{array}{c||c|cc} \hline 0 & 1 & 2 & 2 \\ \hline 0 & 2 & 1 & 2 \\ 0 & 2 & 2 & 1 \\ \hline \hline \end{array}$$

\item  When $i=1$, we have the following sublist. $$\begin{array}{c||c|cc} \hline \hline 1 & 2 & 2 & 0 \\ 1 & 2 & 1 & 1 \\ 1 & 2 & 0 & 2 \\ \hline 1 & 1 & 1 & 2 \\ 1 & 1 & 2 & 1 \\ \hline 1 & 0 & 2 & 2 \\ \hline \hline \end{array}$$

\item  When $i=2$, our list finishes with the following sublist. $$\begin{array}{c||c|cc} \hline \hline 2 & 0 & 1 & 2 \\ 2 & 0 & 2 & 1 \\ \hline 2 & 1 & 2 & 0 \\ 2 & 1 & 1 & 1 \\ 2 & 1 & 0 & 2 \\ \hline 2 & 2 & 0 & 1 \\ 2 & 2 & 1 & 0 \end{array}$$

\end{itemize}

These sublists combine to produce our final list, shown in Figure \ref{GCEx}.

\begin{figure}
$$\begin{array}{c||c|cc}
    0 & 1 & 2 & 2 \\
    \hline
    0 & 2 & 1 & 2 \\
    0 & 2 & 2 & 1 \\
    \hline
    \hline
    1 & 2 & 2 & 0 \\
    1 & 2 & 1 & 1 \\
    1 & 2 & 0 & 2 \\
    \hline
    1 & 1 & 1 & 2 \\
    1 & 1 & 2 & 1 \\
    \hline
    1 & 0 & 2 & 2 \\
    \hline
    \hline
    2 & 0 & 1 & 2 \\
    2 & 0 & 2 & 1 \\
    \hline
    2 & 1 & 2 & 0 \\
    2 & 1 & 1 & 1 \\
    2 & 1 & 0 & 2 \\
    \hline
    2 & 2 & 0 & 1 \\
    2 & 2 & 1 & 0
\end{array}$$
\caption{List Produced by FWM$(3,4,5)$}\label{GCEx}
\end{figure}

In this example, double lines indicate changes in our outer loop, while single lines indicate changes in the secondary loop.

\subsection{Proof of Correctness}

First, we describe more clearly what is happening in Algorithm FWM.  The algorithm is clearly recursive, and the strings are organized so that every string beginning with $i$ comes before every string beginning with $i+1$.  However, within these subsets of our list the ordering is not so simple.  When we consider a sublist of strings that all begin with the same prefix $w_1w_2 \ldots w_\ell$, we can determine whether the list is reversed or not by considering $\sum_{i=1}^\ell w_i$.  If the sum is odd then the list is reversed, and if the sum is even then it is not.  This immediately tells us the ordering of the $(\ell + 1)$st elements in this sublist.

\begin{lem}\label{GCL1}
    The first element of the list FWM$(m,n,k)$ is $0 \cdots 0 r (m-1) \cdots (m-1)$ where $k=q(m-1)+r$ for some $q,r \in \mathbb{Z}$ where $0 \leq r < m-1$.  Define $u_1 = \min \{m-1,k\}$ and $k-u_1 = q'(m-1)+r'$ for some $q',r' \in \mathbb{Z}$ with $0 \leq r' < m-1$.  Then the last element of the list is $$\mathbf{u} = \left\{
                                             \begin{array}{ll}
                                               u_1(m-1) \cdots (m-1)r'0 \cdots 0, & \hbox{if $u_1$ is even;} \\
                                               u_10\cdots 0 r'(m-1) \cdots (m-1), & \hbox{if $u_1$ is odd.}
                                             \end{array}
                                           \right.$$
\end{lem}
\begin{proof}
    It is clear from the algorithm that any list always starts with the minimum string in lexicographic order.  Thus the list must start with $$0 \cdots 0 r (m-1) \cdots (m-1).$$

To find the last element of the list, we proceed by induction on $n$.  For the base case, we consider when $n=1$.  When $n=1$, clearly there is only one string:  $k$.  Note that this agrees with our definition of $\mathbf{u}$.

Before considering the two cases of $u_1$ either odd or even, we note that if $u_1=k$ then the string must be $u_1 0 \cdots 0$ and the lemma is satisfied.  So we may now assume the $u_1 = m-1$.

When $u_1$ is even, we are searching for the last element of the list $$u_1 \oplus \hbox{FWM}(m,n-1,k-u_1).$$ Note that $u_1$ even implies that $m-1$ is even.  In this case, we know that the second letter is $u_2 = \min \{m-1, k-(m-1)\}$.  As before, if $u_2 = k-(m-1)$, then the only string possible is $(m-1)(k-m+1)0 \cdots 0$, which meets the requirements.  So we assume that $u_2 = m-1$, which is even.  So now we know that the last element of the list must be $$u_1 \oplus (m-1) \cdots (m-1)r'' 0 \cdots 0 = (m-1) \cdots (m-1) r' 0 \cdots 0,$$ where $q'', r'' \in \mathbb{Z}$ with $0 \leq r'' < m-1$ so that $k-2(m-1) = q''(m-1)+r''$.

When $u_1$ is odd, we are searching for the last element of the list $$u_1 \oplus \hbox{\textsf{Rev}}(\hbox{FWM}(m,n-1,k-u_1)),$$ which is the same as searching for the first element of the list $$u_1 \oplus \hbox{FWM}(m,n-1,k-u_1).$$  By the first part of the claim this is given by $u_1 0 \cdots 0 r (m-1) \cdots (m-1)$ where $k-u_1 = q(m-1)+r$ for $0 \leq r < m-1$.
\end{proof}

Using this lemma, we are able to deduce the following pair of corollaries.

\begin{cor}\label{GCC1}
    For all $m,n,k$, the first element of FWM$(m,n,k)$ and the first element of FWM$(m,n,k-1)$ differ in exactly one position.
\end{cor}
\begin{proof}
    By Lemma \ref{GCL1}, we know that the first element of FWM$(m,n,k-1)$ is $$0 \cdots 0 r(m-1) \cdots (m-1)$$ with $0 \leq r < m-1$.  Then the first element of FWM$(m,n,k)$ must be $$0 \cdots 0 (r+1) (m-1) \cdots (m-1)$$ where $1 \leq r+1 < m$.  These clearly only differ in one position.
\end{proof}

\begin{cor}\label{GCC2}
    For all $m,n,k$, the last element of FWM$(m,n,k)$ and the last element of FWM$(m,n,k-1)$ differ in exactly one position.
\end{cor}
\begin{proof}
    We proceed by induction on $n$.  If $n=1$, then the two strings must differ in exactly one position.  We now assume that $n>1$.  Define $u_1 = \min \{m-1,k\}$.  By Lemma \ref{GCL1}, the last element of FWM$(m,n,k)$ is $$\mathbf{u} = \left\{
                                             \begin{array}{ll}
                                               u_1(m-1) \cdots (m-1)r'0 \cdots 0, & \hbox{if $u_1$ is even;} \\
                                               u_10\cdots 0 r'(m-1) \cdots (m-1), & \hbox{if $u_1$ is odd.}
                                             \end{array}
                                           \right.$$
As before, if $u_1 = k$ then this string is $\mathbf{u}=k0 \cdots 0$ and the last element of FWM$(m,n,k-1)$ is $(k-1)0 \cdots 0$, which clearly only differs in one position.  If $u_1 = (m-1) \neq k$, then we consider the last element of FWM$(m,n,k-1)$, which we will call $\mathbf{v}$.  In this case we must have $u_1 = v_1$, and so we consider the two substrings $u_2 u_3 \ldots u_n$ and $v_2v_3 \ldots v_n$.  If $u_1$ is even, these are the last elements of the lists FWM$(m,n-1,k-u_1)$ and FWM$(m,n-1,k-u_1-1)$, respectively.  By the induction hypothesis these must differ in exactly one position, and we are done.  If $u_1$ is odd, then our two substrings are the first elements of the lists FWM$(m,n-1,k-u_1)$ and FWM$(m,n-1,k-u_1-1)$, which by Corollary \ref{GCC1} differ in exactly one position.
\end{proof}

Using the lemma and corollaries, we are able to show that Algorithm FWM is correct in the following theorem.

\begin{thm}
    There is a Gray code for $\mathcal{B}_k(m,n)$, the set of $m$-ary words of length $n$ with weight $k$, in which adjacent words differ in exactly two positions.
\end{thm}
\begin{proof}
    We will show by induction that Algorithm FWM produces the desired Gray code.  For the base cases, when $n=1$ the lists are easily constructed.  If $m-1 \geq k$, then we get the list $[k]$, otherwise we have an empty list.

For $n>1$, suppose we want to construct the desired Gray code for $\mathcal{B}_k(m,n)$.  We will show that FWM$(m,n,k)$ is correct.  By the induction hypothesis, for each $i$ from $0$ to $\min \{m-1,k\}$, our sublist $M$ has adjacent elements differing in exactly two positions.  All that remains is to check that this minimal change property is maintained as $i$ increases.

First, when $i$ increases from an odd to an even number, we have the following transition.
\begin{eqnarray*}
    i & \oplus & \hbox{\textsf{Rev}}(\hbox{FWM}(m,n-1,k-i)) \\
    i+1 & \oplus & \hbox{FWM}(m,n-1,k-i-1)
\end{eqnarray*}
 By Corollary \ref{GCC1}, the adjacent elements of these two sublists differ in exactly one position, which together with the leftmost position gives exactly two positions.

When $i$ increases from an even number to an odd number, we have the following transition.
\begin{eqnarray*}
    i & \oplus & \hbox{FWM}(m,n-1,k-i) \\
    i+1 & \oplus & \hbox{\textsf{Rev}}(\hbox{FWM}(m,n-1,k-i-1))
\end{eqnarray*}
Clearly the adjacent elements of these two sublists differ in the first coordinate, so we must check that they only differ in one other position at their meeting point - that is, that the last element of FWM$(m,n-1,k-i-1)$ differs from the last element of FWM$(m,n-1,k-i)$ in at most one position.  To prove this, we use Corollary \ref{GCC2}.
\end{proof}

\section{Overlap Cycles}\label{OCycles}

For a set $S$ of strings of length $n$ and an integer $s$ with $1 \leq s \leq n-1$, an \textbf{$s$-overlap cycle}, or \textbf{$s$-ocycle}, is an ordering of $S$ such that string $x_1x_2 \ldots x_n$ is followed by string $y_1y_2 \ldots y_n$ only if $x_{n-s+i} = y_i$ for $i \in [s]$.  Note that this definition requires that each element of $S$ appear exactly once, and the ordering is cyclic, i.e. the last element of $S$ and the first element of $S$ must overlap in the specified manner.  In other words, an $s$-overlap cycle is a cyclic Gray code that allows string $x$ to follow string $y$ if and only if the last $s$ letters of $y$ (the $s$-suffix) are the same as the first $s$-letters of $x$ (the $s$-prefix).

Given a set $S$ of strings, the standard approach for proving the existence of such an ordering is to use a \textbf{transition digraph}.  The vertices of the transition digraph represent strings of length $s$ (the \textbf{overlaps}), while the edges represent strings of length $n$ (the objects).  For a given string $x=x_1x_2 \ldots x_n$, the \textbf{$s$-prefix} of $x$ is $x^{s-}=x_1x_2 \ldots x_s$ and the \textbf{$s$-suffix} is $x^{s+}=x_{n-s+1}x_{n-s+2} \ldots x_n$.  Let $D$ be the digraph defined by setting $$V=\{w \mid w = x^{s-} \hbox{ or } w = x^{s+} \hbox{ for some } x \in S\}$$ and $$E = \{(w,v) \mid w = x^{s-} \hbox{ and } v = x^{s+} \hbox{ for some } x \in S\}.$$  Then an Euler tour (closed walk that contains every edge exactly once) in $D$ corresponds to an $s$-ocycle for $S$.  To prove the existence of an Euler tour, we use the following well-known result from graph theory.

\begin{thm}\label{euler}
	\emph{(\cite{West}, p. 60)}  A directed graph $G$ is eulerian if and only if it is both balanced and weakly connected.
\end{thm}

\subsection{Existence of Overlap Cycles}

In the paper introducing overlap cycles, Godbole, et al, prove the following general theorem.

\begin{thm}\label{GodboleGen}
    \emph{\cite{Godbole}}  For all $n,m,s \in \mathbb{Z}^+$, $1 \leq s \leq n-1$, there exists an $s$-ocycle for $\mathcal{B}(m,n)$.
\end{thm}

If instead of considering all words over an alphabet of size $M$ we consider words over a fixed multiset, we have the following theorem.

\begin{thm}\label{perms}
    \emph{\cite{Horan}}  Let $n,s \in \mathbb{Z}^+$ with $1 \leq s \leq n-1$ and let $M$ be a multiset of size $n$.  There exists an $s$-ocycle on permutations of $M$ if and only if $n-s > \hbox{gcd}(n,s)$.
\end{thm}

In order to find some middle ground between Theorems \ref{GodboleGen} and \ref{perms}, we consider fixed-weight $m$-ary words.  That is, words over the alphabet $\{0, 1, \ldots, m-1\}$ such that the sum of the letters is some predetermined constant.  For fixed-weight $m$-ary words, we note that the following sets contain only one string or a set of strings with the same fixed multiset, hence by Theorem \ref{perms} there always exists an $s$-ocycle for:
\begin{itemize}
    \item  $\mathcal{B}_0(m,n) = \{0^n\}$,
    \item  $\mathcal{B}_1(m,n) = \{0^t10^{n-t-1} \mid t \in [n]\}$,
    \item  $\mathcal{B}_2(2,2) = \{11\}$,
    \item  $\mathcal{B}_{(m-1)n}(m,n) = \{(m-1)^n\}$, and
    \item  $\mathcal{B}_{(m-1)n-1}(m,n) = \{(m-1)^t(m-2)(m-1)^{n-t-1} \mid t \in [n]\}$.
\end{itemize}

\begin{thm}
    Fix $n,m,k,s \in \mathbb{Z}^+$ such that:
    \begin{itemize}
        \item  $1 \leq s \leq n-1$, and
        \item  $1 < k < (m-1)n-1$.
    \end{itemize}
    Then there is an $s$-ocycle for $\mathcal{B}_k(m,n)$ if and only if $n-s > \hbox{gcd}(n,s)$.
\end{thm}

\begin{proof}
    Suppose that $n-s > d$ where $d = \hbox{gcd}(n,s)$.  Construct the transition graph $G=G_s(m,n)$ with vertices representing $s$-prefixes and $s$-suffixes of elements in $\mathcal{B}_k(m,n)$ and edges representing elements of $\mathcal{B}_k(m,n)$, traveling from prefix to suffix.  First, at each vertex $X$ in $G$, for each $(n-s)$-suffix $Y$ for $X$ we have an out-edge representing the string $XY \in \mathcal{B}_k(m,n)$. Note that also $YX \in \mathcal{B}_k(m,n)$, and is represented by an in-edge to $X$.  This gives a bijection between in- and out-edges at $X$, and hence $G$ is balanced.

    There exist some $r,q \in \mathbb{Z}$ such that $0 \leq r < m-1$ with $k = q(m-1)+r$.  Define the minimum element $V$ of $\mathcal{B}_k(m,n)$ to be the minimum lexicographically, i.e. $V = 0 \cdots 0 r (m-1) \cdots (m-1)$.  Then the minimum vertex in $G$, call it $V^{s-}$, is the $s$-prefix of $V$.  Let $X^{s-} = x_1x_2 \ldots x_s$ be an arbitrary vertex in $G$.  Let $X \in \mathcal{B}_k(m,n)$ be an $m$-ary string with $s$-prefix $X^{s-}$.  We will prove that $G$ is connected by illustrating a path from $X^{s-}$ to $V^{s-}$.

    Using Theorem \ref{perms}, all permutations of the string $X$ are connected in $G$, so we may assume that $X = x_1x_2 \ldots x_n$ is ordered so that $x_1 \leq x_2 \leq \ldots \leq x_n$.  Compare $X$ with $V$, and let $i$ be the left-most index in which they disagree.  In other words, we have $x_1 = v_1$, $x_2 = v_2$, and so on until $x_{i-1} = v_{i-1}$, but $x_i \neq v_i$.  Since $V$ is the minimum string lexicographically, this implies that $x_i > v_i$.  Since $X$ and $V$ both have weight $k$, there must be some index $j > i$ such that $x_j < v_j$.  Using Theorem \ref{perms} again, we find a path to the $s$-prefix of the string $$X' = x_1 x_2 \ldots x_{i-1} x_{i+1} x_{i+2} \ldots x_{j-1} x_{j+1} x_{j+2} \ldots x_n x_i x_j.$$  Note that this $s$-prefix does not contain the last two letters, namely $x_i$ and $x_j$, so the $s$-prefix also has an out-edge representing the string $$X'' = x_1 x_2 \ldots x_{i-1} x_{i+1} x_{i+2} \ldots x_{j-1} x_{j+1} x_{j+2} \ldots x_n (x_i-1) (x_j+1).$$  Applying Theorem \ref{perms} again, we can find a path to the $s$-prefix of the string $$x_1 x_2 \ldots x_{i-1} (x_i-1) x_{i+1} x_{i+2} \ldots x_{j-1} (x_j+1) x_{j+1} x_{j+2} \ldots x_n.$$  This $s$-prefix is now closer to the minimum vertex, and repeating the process we will eventually arrive at $V^{s-}$.  Since $X^{s-}$ was an arbitrary starting vertex in $G$, $G$ is connected and hence eulerian by Theorem \ref{euler}.

    For the converse, suppose that $n-s = \hbox{gcd}(n,s)$.  Then, as discussed in \cite{Horan}, rotations of a string $X$ partition it into blocks of length $d = \hbox{gcd}(n,s)$.  Since $X$ has fixed weight and we may only reorder elements within blocks but not swap elements between blocks, each block has a fixed weight.  In other words, if $X = Y_1 Y_2 \ldots Y_\ell$ is $X$ partitioned into $d$-blocks with block $Y_i$ having weight $w_i$, then no manipulation of $X$ can alter the values $w_i$ corresponding to each $Y_i$.  Hence if we can always produce two strings $X,Z \in \mathcal{B}_k(m,n)$ so that the block sequences for $X$ and $Z$ are not simply rotations of each other, then we are done.

    First note that $k \geq 2$.  There are integers $q,r$ such that $0 \leq r < m-1$ and $k = (m-1)q+r$.  Define two strings:  $$A=0^{n-q-1}r(m-1)^q \hbox{ and } B=10^{n-q-2}r(m-1)^{q-1}(m-2),$$ with block sequences $A_1A_2 \ldots A_{\ell}$ and $B_1B_2 \ldots B_{\ell}$ respectively.  We will show through several cases that these two block sequences are distinct and are not identical or rotations of each other.
    \begin{enumerate}
        \item  If $n-q-1\geq d$ and $q \geq d$:

        In this case we know that:
        \begin{itemize}
            \item  $w(A_1)=0$,
            \item  $w(A_\ell)=(m-1)d$,
            \item  $w(B_1)=1$, and
            \item  $w(B_\ell)=(m-1)d-1$.
        \end{itemize}

        In this case our two block weight sequences are as follows, where $x=(m-1)d$ and $y=r+(m-1)j$ for some $j \leq d$.
        $$\begin{array}{lc|ccc|c|ccc|c}
            A= & 0 & 0 & \cdots & 0 & y & x & \cdots & x & x \\
            \hline
            B= & 1 & 0 & \cdots & 0 & y & x & \cdots & x & x-1
        \end{array}$$
        Note that if the weight sequence for $A$ contains a string of $\alpha$ consecutive 0's, then the sequence for $B$ contains a string of $\alpha-1$ consecutive 0's.  Thus these two strings can only be rotations of each other if $\alpha=1$ (or only the first block has weight $0$) and $x-1=0$.  Thus our block weight sequences are as follows.
        $$\begin{array}{lc|c|ccc|c}
            A= & 0 & 1 & 1 & \cdots & 1 & 1 \\
            \hline
            B= & 1 & 1 & 1 & \cdots & 1 & 0
        \end{array}$$
        Hence we have $(m-1)d=1$, which implies that $m=2$ and $d=1$, so given the block sequences we must have $k=n-1$ in order for the sequences to be rotations of each other.  However the initial conditions require that $k<n-1$, so these two sequences will never be rotations.

        \item  If $n-q-1 \geq d$ and $q < d$:

        In this case we know that:
        \begin{itemize}
            \item  $w(A_i)=0$ for $1 \leq i < \ell$,
            \item  $w(A_\ell)=k$,
            \item  $w(B_1) = 1$,
            \item  $w(B_i)=0$ for $2 \leq i < \ell$, and
            \item  $w(B_\ell)=k-1$.
        \end{itemize}

        In this case our two block weight sequences are as follows.
        $$\begin{array}{lc|ccc|c}
            A= & 0 & 0 & \cdots & 0 & k \\
            \hline
            B= & 1 & 0 & \cdots & 0 & k-1
        \end{array}$$
        These two block weight sequences are only rotations of each other if $k=1$, which is not allowed by the hypotheses.

        \item  If $n-q-1 < d$ and $q \geq d$:

        In this case we know that:
        \begin{itemize}
            \item  $w(A_1)=r+j(m-1)$ for some $0 \leq j < d$,
            \item  $w(A_i)=d(m-1)$ for $2 \leq i \leq \ell$,
            \item  $w(B_1) = r+j(m-1)+1$ for some $0 \leq j < d$,
            \item  $w(B_i)=d(m-1)$ for $2 \leq i < \ell$, and
            \item  $w(B_\ell)=d(m-1)-1$.
        \end{itemize}

        In this case our two block weight sequences are as follows, where $x=(m-1)d$ and $y=r+(m-1)j$ for some $j \leq d$.
        $$\begin{array}{lc|ccc|c}
            A= & y & x & \cdots & x & x \\
            \hline
            B= & y+1 & x & \cdots & x & x-1
        \end{array}$$
        These two block weight sequences are only rotations of each other if $y=(m-1)d-1$.  However in this case we must have $k=(m-1)n-1$, which is not allowed by the initial conditions.

        \item  If $n-q-1 < d$ and $q < d$:

        In this case we have that $A = 0^{n-q-1}r(m-1)^q$, but then we must have $(n-q-1)+1 + q \leq 2d-2$, or $n \leq 2d-2$, which is not possible since we must have $d \leq \frac{n}{2}$.
    \end{enumerate}
\end{proof}

Next we consider $m$-ary words with weights belonging to some fixed range.  That is, we consider $s$-ocycles on the set of \textbf{weight-range} $m$-ary words, defined as ($p<q$):
$$\mathcal{B}_p^q(m,n) = \{x=x_1x_2 \ldots x_n \in \mathcal{B}(m,n) \mid p \leq \hbox{wt}(x) \leq q\}.$$  This small amount of room in weights allows us to prove a much stronger result.

\begin{thm}
    For $m,n,s,p,q \in \mathbb{Z}$, with $1 \leq s<n$ and $0 \leq p < q \leq (m-1)n$, there exists an $s$-ocycle for $\mathcal{B}_p^q(m,n)$.
\end{thm}
\begin{proof}
    Construct the transition digraph $\overrightarrow{T}$.  First, note that any $s$-prefix is also an $s$-suffix, so $\overrightarrow{T}$ is balanced.

    To show connectivity, we denote a minimum vertex and show a path always exists from an arbitrary vertex to the minimum vertex.  We will define the minimum vertex of the digraph as follows.  Write $p= a(m-1)+b$ for some $0 \leq b < m-1$.  Then the minimum vertex $v$ is defined to be the $s$-prefix of the string $0 \cdots 0 b (m-1) \cdots (m-1)$.  In other words, $v$ is the $s$-prefix of the minimum lexicographic string in the set.

    Now we will prove that the digraph is weakly connected.  Let $x=x_1x_2 \ldots x_n$ and consider the vertex given by the $s$-prefix, $x_1x_2 \ldots x_s$.  As usual, rotations of $x$ in the digraph correspond to a partition of $x$ into blocks of size $d=\hbox{gcd}(n,s)$.  Note that through rotations, we can always ``hide" at least one block of $x$ in the $(n-s)$-suffix of the string.  Thus we can modify \textit{at least} one block at a time; however, as seen in the previous proof we cannot always assume it is possible to modify \textit{more} than one block at a time without additional requirements.  Through this process, we may assume that a valid operation is one that modifies exactly one block in each step.  Thus our path construction consists of the following two parts.

    \begin{enumerate}
        \item  If $\hbox{wt}(x) > p$:

        Consider the block partition of $x$, given by $x = X_1X_2 \ldots X_\delta$ where $\delta = n/d$.  Suppose that block $X_i$ has maximum weight.  Then since $\hbox{wt}(x) > p$, we must have $\hbox{wt}(X_i) > 0$, so we rotate until block $X_i$ is in the $(n-s)$-prefix of $x$ (i.e. $X_i$ does not appear in the $s$-prefix).  At this point we can modify the block $X_i$ and replace it with some block $X_i'$ with smaller weight.  If we are now considering an $n$-string with weight $p$ continue to step 2.  Otherwise, repeat step 1.

        \item  If $\hbox{wt}(x) = p$:

        Compare the block weights of $x$ and $v$.  If the block weight sequences are equal (i.e. $\hbox{wt}(X_i) = \hbox{wt}(V_i)$ for all $i \in [\delta]$), then repeatedly sort one block of $x$ at a time to arrive at $v$. Otherwise, the block weight sequences differ in at least two positions.  That is, there exists $i,j$ such that $\hbox{wt}(X_i) <  \hbox{wt}(V_i)$ and $\hbox{wt}(X_j) > \hbox{wt}(V_j)$.  Then we first increase the weight of $X_i$ by one (which must be possible since $V_i$ is on the same $m$-ary alphabet), and afterwards we decrease the weight $X_j$ by one (which must be possible since $\hbox{wt}(X_j) > \hbox{wt}(V_j)$).  Now either the block weight sequences for $x$ and $v$ are identical and we are done, or we repeat step 2 until that happens.  Note that at each step we are decreasing the measure $\sum_{i=1}^\delta |\hbox{wt}(X_i)-\hbox{wt}(V_i)|$ by two, and so the process will terminate when that measure reaches zero.
    \end{enumerate}
\end{proof}

\section{Future Work}\label{FW}

We leave the following question as a possible future direction for this area of research.  Is there a Gray code for weight-range $m$-ary strings such that the strings differ in exactly one position?  Note that this would solve the famous middle-levels problem \cite{MLP}.

\bibliographystyle{amsplain}

\end{document}